\documentclass[11pt,a4paper]{article}

\usepackage[a4paper, total={5.8in, 9in}]{geometry}

\usepackage[english]{babel}
\usepackage{mdframed}

\usepackage[OT4]{fontenc}

\usepackage{bbm}
\usepackage{enumerate}
\usepackage{amssymb}
\usepackage{amsmath}
\usepackage{alltt}
\usepackage{amsthm}
\usepackage{amsfonts}
\usepackage{mathtools}
\usepackage{latexsym}
\usepackage{xspace}

\usepackage{amstext}

\usepackage[dvipsnames]{xcolor}

\usepackage{fullpage}
\usepackage{color}
\usepackage{graphicx}
\usepackage{graphics}
\usepackage{url}

\usepackage{tikz}
\usetikzlibrary{fit}
\usetikzlibrary{arrows.meta}
\usepackage{thm-restate}
\usepackage{complexity}
\usepackage{array}
\usepackage{ifthen}

\usepackage{thmtools}

\usepackage{algorithm}
\usepackage{algpseudocode}
\algtext*{EndWhile}
\algtext*{EndIf}
\algtext*{EndFor}
\algrenewcommand\algorithmicrequire{\textbf{Input:}}
\algrenewcommand\algorithmicensure{\textbf{Output:}}

\usepackage[colorlinks=true, allcolors=blue]{hyperref}

\usepackage[noabbrev,capitalise,nameinlink]{cleveref}

\newcommand{\graphProd}{\leftrightarrows}

\newcommand{\cand}{\mathrm{cand} \xspace}
\newcommand{\constf}{\phi}

\newcommand{\outdom}{\gamma^+}
\newcommand{\indom}{\gamma^-}

\newcommand{\SST}{\mathcal{S}}

\definecolor{bda}{HTML}{00B388}

\newcommand{\Ed}{{\cal E}}
\newcommand{\zz}{z^*}
\newcommand{\chiv}{\vec{\chi}}

\newcommand{\pp}{p}

\newtheorem{theorem}{Theorem}
\newtheorem{lemma}{Lemma}
\newtheorem{conjecture}{Conjecture}
\newtheorem{definition}{Definition}
\newtheorem{claim}{Claim}
\newtheorem{observation}{Observation}
\newtheorem{corollary}{Corollary}
\newtheorem{question}{Question}

\newcommand{\chat}{{GPT-Astra}\xspace}

\newcommand{\Ex}{\mathbb{E}}

\newcommand{\magic}{\gamma^{\star}}

\newcommand{\KK}{4.9108...}
\newcommand{\KKc}{4.9109}

\newenvironment{cproof}
{\begin{proof}
 [Proof.]
 \vspace{-1.5\parsep}
}
{ \end{proof}}

\begin{document}

\title{Elections, Fractional Dichromatic Number and Domination in Tournaments and Digraphs}

\author{Paul Colinot\thanks{G-SCOP, Universit\'e Grenoble Alpes, France. \tt{paul.colinot@grenoble-inp.fr}} \and Alantha Newman\thanks{LIP, CNRS, ENS de Lyon, France. \tt{alantha.newman@ens-lyon.fr}}}

\date{}

\maketitle

\begin{abstract}

The dichromatic number of a digraph is the minimum number of induced
acyclic sets into which its vertex set can be partitioned.  The
fractional dichromatic number is its fractional relaxation.  Bourneuf,
Charbit and Thomass\'e showed that the domination number of a
tournament can be bounded as an (exponential) function of its
fractional dichromatic number.  This naturally leads to the question
of finding a polynomial bound.

We show that this bound is linear.  To do this, we show a connection
between undominated sets in elections and dominating sets in
tournaments with bounded fractional dichromatic number.  Within this
framework, we also connect the well-known question of the existence of
small Condorcet winning sets in elections to domination problems in
digraphs with fractional dichromatic number less than 2.  Thus, we
open a new path to investigating these popular open problems in social
choice from the perspective of graph theory.

\end{abstract}

\section{Introduction}

We use $T$ to denote a digraph with vertex set $V(T)$ and arc set
$A(T)$.  The {\em dichromatic} number of a tournament $T$, denoted
$\chiv(T)$, is the minimum integer $k$ such that the vertex set of $T$
can be partitioned into $k$ transitive subtournaments.  (More
generally, the dichromatic number of a directed graph or digraph is
the minimum number of induced acyclic subgraphs into which its vertex
set can be partitioned.)  The dichromatic number of tournaments has
been studied in the context of structural graph theory
\cite{berger2013tournaments} and from an algorithmic
perspective~\cite{klingelhoefer2024coloring}.

\paragraph{Local-to-Global.}
An interesting local-to-global question, posed as Conjecture 2.6 in
\cite{berger2013tournaments}, is: if the out-neighborhood of each
vertex in a tournament has bounded dichromatic number, then does the
whole tournament have bounded dichromatic number?  (Observe that such
a statement does not hold for graphs and chromatic number, since a
triangle free graph can have unbounded chromatic
number~\cite{mycielski1955coloriage, erdos1959graph,kim1995ramsey},
and in such a graph, the neighborhood of each vertex forms an
independent set.) \cite{harutyunyan2019coloring} answered this
question in the affirmative.

Specifically, we say a tournament is {\em $t$-bounded} if for every
vertex $v$, the subtournament of $T$ induced by the set of
out-neighbors of $v$ has dichromatic number at most $t$.  The
following theorem states that the dichromatic number of a $t$-bounded
tournament is bounded as a function of $t$.  The function itself is
left unstated in the paper, but is clearly exponential.
\begin{theorem}[\cite{harutyunyan2019coloring}]\label{thm:localtoglobal}
There is a function $f$ such that every $t$-bounded tournament $T$
satisfies $\chiv(T) \leq f(t)$.
  \end{theorem}
Let $\outdom(T)$ denote the size of the minimum (out) dominating set
in $T$, where $S \subset V(T)$ is a dominating set of $T$ if every
vertex in $T$ either belongs to $S$ or has at least one in-neighbor in
$S$.  The original proof of \cite{harutyunyan2019coloring} showed,
somewhat indirectly, that the domination number of a $t$-bounded
tournament is bounded, which yields a proof of Theorem
\ref{thm:localtoglobal}.  We state this formally.
\begin{observation} \label{obs:tBounded}
If there is a function $g$ such that every $t$-bounded tournament $T$
satisfies $\outdom(T) \leq g(t)$, then $f(t) = t \cdot g(t)$.
\end{observation}

It has also been observed that if $T$ is $t$-bounded, then
$\chiv_f(T)$ is bounded as a function of $t$, where $\chiv_f(T)$
denotes the fractional dichromatic number, which is the fractional
analogue of dichromatic number and will be defined shortly.  This
proof follows easily from the fact that the fractional domination
number of a tournament is at most
2~\cite{Harutyunyan_Le_Newman_Thomasse_2018}.
\begin{observation}\label{obs:fractBounded}
Every $t$-bounded tournament $T$ satisfies $\chiv_f(T) \leq 2t$.
\end{observation}
Recently, \cite{bourneuf2025dense} proved that bounded fractional dichromatic number implies bounded domination.
\begin{theorem}[\cite{bourneuf2025dense}]\label{thm:main}
There is a function $h$ such that if tournament $T$ satisfies $\chiv_f(T) \leq c$, then $\outdom(T) \leq h(c)$.
\end{theorem}
Their function $h$ is exponential and it is natural to investigate whether or not there exists a polynomial function for $h$.  
We present the first polynomial bounds for the function $h$ in Theorem \ref{thm:main}.
One of our main findings is that this function is linear.
\begin{theorem}\label{thm:main-new}
  If tournament $T$ satisfies $\chiv_f(T) \leq c$, then $\outdom(T) \leq  \lceil \KKc c \rceil$.
\end{theorem}
This proof follows from a connection between fractional dichromatic number, domination in tournaments and elections.  Before we go into more detail, we present some background on dominating sets and elections.

\paragraph{Elections and Condorcet Winning Sets.}

An {\em election} consists of $m$ rankings
on a set of $n$ candidates.  For a voter $v$, we say that {\emph{candidate $a$ beats candidate $b$}} if candidate $a$ appears before $b$ in the ranking of voter $v$.  Sometimes this is written as $a \succ_v b$.
An interesting notion, introduced by \cite{elkind2015condorcet} is that of a {\emph{Condorcet winning set}}.  
\begin{definition}[\cite{elkind2015condorcet}]
  Given an election on $n$ candidates $C$ and $m$ voters, a set of candidates $W$ is a {\em Condorcet winning set (CWS)} if no candidate $a \in C\setminus{W}$ is ranked above all members of $W$ by more than half the voters.
 \end{definition}
\cite{elkind2015condorcet} showed that every election has a Condorcet winning set of size $O(\log{n})$, proved a lower bound of three, and asked if there exists a positive integer $k$ such that every election has a Condorcet winning set of size $k$.

In this paper, we will use {\emph{weighted elections}}.  A weighted
election is a set of $m$ voter rankings on $n$ candidates, where each
voter $v$ has weight $w_v \geq 0$ and $\sum_{v} w_v = 1$.  When the
weights are rational, a weighted election can be represented by an
equivalent unweighted election.  The definition of a Condorcet winning
set can be extended to a weighted election.
\begin{definition}
In a weighted election, each voter $v$ provides a ranking of $m$
candidates, counted with weight $w_v$.  A set of candidates $W$ is a
{\em Condorcet winning set} if no challenger candidate $a \in
C\setminus{W}$ beats all members of $W$ in voter ranking consisting of
strictly more than half the total weight.
\end{definition}

In a major advance, \cite{charikar2025six} proved that every election
has a Condorcet winning set of size six.  Recently,
\cite{conf/soda/SongNL26} improved this bound to five.  Alternate
explanations and proofs can be found in
\cite{charikar2026expositioncandidatessufficemajority} and
\cite{prasanna2026blog}.
\begin{theorem}[\cite{conf/soda/SongNL26,charikar2026expositioncandidatessufficemajority,prasanna2026blog}]\label{thm:cws5}
Every election has a Condorcet winning set of size five.
  \end{theorem}
The exact bound is still open.
\begin{conjecture}[\cite{charikar2025six}]\label{conj:cws3}
Every election has a Condorcet winning set of size three.
  \end{conjecture}
The following definition is a generalization of Condorcet winning sets.
\begin{definition}[\cite{elkind2015condorcet}]
A committee $S$ is $\alpha$-undominated if for all candidates $a \notin S$, at most an $\alpha$-fraction of (weighted) voters prefer $a$ over each member of $S$.\footnote{Sometimes this definition uses ``less than'' in place of ``at most''.  Here, we use the latter so that a $\frac{1}{2}$-undominated set and CWS coincide.}
  \end{definition}


The current best bound on $\alpha$-undominated committees in elections
is the following.
\begin{theorem}[\cite{conf/soda/SongNL26, charikar2026expositioncandidatessufficemajority}]\label{thm:snl26}
Every election has a committee $S$ of at most $k$ candidates such that
any candidate $a \notin S$ is preferred to $S$ by at most a
$\frac{\KK}{k}$-fraction of the voters.
\end{theorem}
The factor of $f = \KKc$ improved upon factors of $9.82...$
\cite{charikar2025six} and $16$
\cite{jiang2020approximately,cheng2020group}.  For more background on
Condorcet winning sets and $\alpha$-undominated sets, we refer the
reader to the comprehensive introduction of \cite{charikar2025six}.

\paragraph{Connecting Fractional Dichromatic Number, Domination and Elections.}

In our proof of Theorem \ref{thm:main-new}, we use a solution for the
fractional dichromatic number to construct an election.  Before we
explain this further, let us formally define this notion.  The {\em
  fractional dichromatic} number of a digraph $D$ equals the objective
value of the following linear program, which is a fractional
relaxation of the dichromatic number.  Let $\SST$ denote the vertex
subsets that induce transitive subdigraphs in $D$ (i.e., if $S \in
\SST$, then $D[S]$ is acyclic).  The variable $z_{S}$ denotes the
fraction of the induced acyclic subdigraph $D[S]$ used to cover
vertices contained in $S$.
\begin{align*}
\min \sum_{S \in \SST} z_{S} &\nonumber\\
\sum_{S: v \in S} z_{S} \geq 1 &  \quad \forall v \in V(D) \nonumber\\
z_{S} \geq 0. &\nonumber \tag{P}\label{LP-frac}
\end{align*}
We observe that the support of a basic feasible solution $\{z_S\}$ has
size at most $n$.  There is a natural way to use a solution to
\eqref{LP-frac} to construct an election.
\begin{definition}\label{def:wpe}
Let $D$ be a digraph such that $\chiv_f(D) = c$, implying there is a
solution to \eqref{LP-frac} such that $\sum_{S \in \SST} z_S = c$.  We
can construct a {\em weighted-prefix election} as follows.  For each
$S \in \SST$, we define a voter whose ranking has $S$ as a prefix,
ordering vertices in $S$ so that all arcs in $D[S]$ are forward, and
the remaining vertices in $V(D)\setminus{S}$ afterwards in an
arbitrary order.  This voter $v$ has weight $w_v =
z_S/c$.  \end{definition}

From a solution to \eqref{LP-frac} on a tournament $T$, we construct a
weighted-prefix election.  We show that dominating sets in $T$ can be
obtained by finding $\alpha$-undominated sets in the corresponding
weighted-prefix election.  Hence we obtain linear bounds on the
domination number of tournaments in terms of its fractional
dichromatic number by applying known results for $\alpha$-undominated
sets such as Theorem \ref{thm:snl26}.  We then observe that dominating
sets in these elections are also closely related to Condorcet winning
sets: if every election has a Condorcet winning set of size $k$, then
a tournament with fractional dichromatic number less than 2 has a
dominating set of size $k$.

This could be a potential path to disprove Conjecture \ref{conj:cws3},
but it seems likely that such tournaments always have dominating sets
of size 2.  So we turn to general digraphs and define a simple
variation of domination called {\emph{magic domination}}.  In this
variant, for an oriented graph (a directed graph with no directed
2-cycles), we are allowed to dominate a vertex by an in-neighbor or a
non-neighbor (i.e., via an ``invisible'' link).  We show that proving
the existence of magic dominating sets of size $k$ in a digraph $D$
with fractional dichromatic number less than 2 is equivalent to
proving existence of Condorcet winning sets of size $k$ for any
election.  For general digraphs, we show an equivalence between size
$k$ dominating sets and size $k$ $\alpha$-undominated sets.

In an initial version of this paper, we gave a quasilinear bound for
Theorem \ref{thm:main-new}.  This proof, found by AI, is closely
related to the book proof that every election has a Condorcet winning
set of size five~\cite{prasanna2026blog}.  However, it does not rely
on Theorem \ref{thm:cws5} or \ref{thm:snl26}.  After posting that
preprint, \cite{kim2026tournamentdichromaticnumber10} presented a
clever way to modify this proof to obtain a linear bound.  Using their
idea, and the equivalence between magic domination and
$\alpha$-undominated sets, we obtain a new proof of Theorem
\ref{thm:snl26}.  Finally, in Section \ref{sec:graph-theory}, we
catalogue the bounds on domination in tournaments and magic domination
in digraphs as a function of their fractional dichromatic number.

\paragraph{Concurrent Work.}
We obtained the quasilinear bound for Theorem \ref{thm:main} from
GPT-Astra (see Theorem \ref{thm:quasilinear}), which we published in
an intial arXiv version of this paper.  At first, we were surprised by
this bound, but then we realized the connection to elections,
presented in Section \ref{sec:tournaments}, which lead to a linear
bound, using different ideas from those used in the initial
quasilinear bound.  We wanted to explore this further, so we did not
rush to update our paper on the arXiv.  A few weeks later, in
independent work, \cite{kim2026tournamentdichromaticnumber10}
presented a clever way to improve the quasilinear proof to linear,
obtaining an independent proof of a linear bound for Theorem
\ref{thm:main}.  We then used their ideas in Section \ref{sec:magic}
to obtain another proof of Theorem \ref{thm:snl26}.

\section{Domination on Tournaments} \label{sec:tournaments}

In this section, we give a proof of Theorem \ref{thm:main-new}.
First, we present an implication of Theorem \ref{thm:snl26} for weighted-prefix elections.  We will use the following lemma several times in this paper.

\begin{lemma}\label{clm:common}
  Let $D$ be a digraph and let vertex $u \in V(D)$ be a common in-neighbor of $W \subset V(D)$.  Then in a weighted prefix election for $D$, if we choose a voter $v$ with probability $w_v$, 
  $$\Pr[u \succ_v W] \geq \frac{1}{c}.$$
  \end{lemma}

\begin{proof}
The weighted-prefix election for $D$ is constructed using a solution for \eqref{LP-frac}.  For every voter that has an $S \ni u$ as a prefix, $u$ precedes $W$.  If $S \not\ni u$, then $u$ could still precede $W$.  So we have
    $$\frac{1}{c} 
  \leq \sum_{S\in\SST: u \in S} \frac{z_S}{c} \leq \sum_{v:u \succ_v W} w_v = \Pr[u \succ_v W].$$\end{proof}

The next theorem lets us transfer (known and future) results on
$\alpha$-undominated sets directly to an upper bound of the domination
of tournaments.  The current best bound on $\alpha$-undominated
committees in elections is stated in Theorem \ref{thm:snl26}.  Hence,
Theorem \ref{thm:main-new} follows as a corollary of Theorem
\ref{th:linear} and Theorem \ref{thm:snl26}.

\begin{theorem}\label{th:linear}
  Let $\constf >0$ be a constant.  Suppose that for every integer $k
  \geq 1$, every weighted election on at least $k$ candidates has a
  committee $S$ of size $k$ such that, for each $u \notin S$, the
  voters who rank $u$ strictly above every member of $S$ have total
  weight less than $\phi/k$.  Then, for any tournament $T$ with
  fractional dichromatic number $c$, there is a dominating set of size
  at most $\lceil{\constf \cdot c}\rceil$.
\end{theorem}

\begin{proof}
Assume $\outdom(T) \geq 2$ and set $k=\outdom(T)-1$.  Construct a
weighted prefix election (see Definition \ref{def:wpe}) for $T$.  By
assumption, for this election, there is a committee $S$ of size $k$
such that, for every candidate $u \notin S$,
    \begin{eqnarray*}
 \Pr[u \succ_v S]  < \frac \constf k. 
    \end{eqnarray*}    
Moreover, since $S$ has size $k$, it is not a dominating set of $T$.
So there is a vertex $u$ that is not dominated by $S$ in $T$; in other words, $u$ is a common in-neighbor of $S$ in $T$.  By Lemma \ref{clm:common}, we have
\begin{eqnarray*}
\Pr[u \succ_v S] \geq \frac1c.
\end{eqnarray*}
Thus, we have
\begin{eqnarray*}
        \frac1c < \frac \constf k.
    \end{eqnarray*}
We conclude that $k\leq \lceil{\constf \cdot c}\rceil-1$ and $\outdom(T)\leq \lceil{\constf \cdot c}\rceil$.
\end{proof}

\subsection{Bounds when Fractional Dichromatic Number is Small}

When $\chiv_f(T)$ is small, we can prove better bounds than those stated in 
Theorem \ref{thm:main-new}.  This is related to the fact that the existence 
of $\frac {\KK}k$-undominated sets is currently the best-known bound in general for sets of size $k$. However, it is possible to obtain a better upper bound for small values of $k$.

For $k \geq 2$, set $\alpha_k =1-(k-1)k^{-k/(k-1)}$ and $c_k = 1/\alpha_k$.
\begin{theorem}[\cite{conf/soda/SongNL26,prasanna2026blog}]\label{th:bound-small-k}
  Every election 
  has an
  $\alpha_k$-undominated committee of size $k$.\footnote{When we say an election has a committee of size $k$, we always assume the election has at least $k$ candidates.}
\end{theorem}

A direct corollary of this theorem is that every election has a Condorcet winning set of size five.  We can follow the proof from \cite{prasanna2026blog} almost exactly to show the following theorem.

\begin{theorem}\label{thm:small-dom}
  If a tournament satisfies $\chiv_f(T) \leq c_k$, then
$\outdom(T) \leq k$.
\end{theorem}

\begin{proof}
Let $T$ be a tournament with fractional dichromatic number $c$.  For
$k \leq \outdom(T)-1$, every set of at most $k$ vertices $X$ has a
common in-neighbor, $h(X)$.  Following the proof in
\cite{prasanna2026blog}, we observe that there is a fixed point
probability distribution $p$, so that if we choose $k$ vertices
$X=\{x_1, \ldots, x_k\}$ from $p$, we also choose $h(X)$ from the
distribution $p$.

Now we construct a weighted-prefix election for $T$.  For a voter $v$,
we fix an interval $[0,1]$ and give each candidate $a$ a block with
width equal to $p(a)$, and we place these blocks in the ranked order
of voter $v$.  If we fix some $t$, and let candidate $\cand_v(t)$ be
the candidate occupying block with value $t$ for voter $v$, where $\cand_v(t)$ is voter $v$'s favorite candidate.
Then we
have
\begin{eqnarray*}
\Pr[h(X) \succ_v X] \leq \Pr_{h(X) \sim p}[h(X) \succ_v \cand_v(t)] + \Pr_{X \sim p^k}[\cand_v(t) \succ_v X].
\end{eqnarray*}
Averaging over all voters, we have that the left-hand side is at least $1/c$; in a weighted-prefix election, a common in-neighbor $h(X)$ of a set $X$ precedes $X$ for at least a $1/c$-fraction of the voters. For the right-hand side, we have
\begin{eqnarray*}
\Pr_{h(X) \sim p}[h(X) \succ_v \cand_v(t)] < 1-t \quad
\text{ and } \quad
\Pr_{h(X) \sim p}[h(X) \succ_v \cand_v(t)] \leq t^k.
\end{eqnarray*}
Setting $t = k^{-1/(k-1)}$ gives the bound $1/c < \alpha_k = 1 - (k-1)k^{-k/(k-1)}$.
Hence, if $1/c \geq \alpha_k$, then $\outdom(T) < k+1$, so $\outdom(T) \leq k$.  \end{proof}

Since $c_5 > 2$, we have the following corollary.
\begin{corollary}
If a tournament $T$ satisfies $\chiv_f(T) \leq 2$, then $\outdom(T) \leq 5$.
\end{corollary}

This approach leads to the following lemma, which provides a weaker bound than Theorem \ref{thm:main-new}, but illustrates the connection between the approach for undominated sets in elections and dominating sets in tournaments.
\begin{lemma}\label{lem:poly-tournois}
If a tournament $T$ satisfies $\chiv_f(T) \leq c$, then $\outdom(T) \leq \lceil{2c \ln{(2c)}}\rceil$.
\end{lemma}

\begin{proof}
Reversing the role of $t$ and $(1-t)$ for convenience, and setting $t = 1/(2c)$, we have
  \begin{eqnarray*}
\frac{1}{c} < t + (1-t)^k \leq t + e^{-tk} = \frac{1}{2c} + e^{-k/(2c)},
  \end{eqnarray*}
  which implies that $$e^{-k/(2c)} > \frac{1}{2c},$$ which implies the claimed bound.
  \end{proof}

When we tried all the thresholds (rather than just $1/(2c)$ in the above proof), we were initially surprised to find out that when $\chiv_f(T) \leq 2$, then $\outdom(T) \leq 5$, since this is the same bound the proof in \cite{prasanna2026blog} gives for Condorcet winning sets.  But then we realized that any bound for Condorcet winning sets transfers to the case of $\chiv_f(T) < 2$ for tournaments.

\begin{lemma}\label{lem:tourn-cws-dom}
If every election has a Condorcet winning set of size $k$, then every tournament with $\chiv_f(T) < 2$ satisfies $\outdom(T) \leq k$.
\end{lemma}

\begin{proof}
  Using the solution for \eqref{LP-frac}, we construct the weighted-prefix election on $T$.  Let $\chiv_f(T) = c$. We claim that a CWS of this election is a dominating set of $T$.  To see this, fix a CWS $W$.  If a vertex $a$ is not dominated by $W$ in $T$, then $a$ is a common in-neighbor of all vertices in $W$.  Hence, by Lemma \ref{clm:common}, we have
  $$\frac{1}{2} < \frac{1}{c} 
  \leq \Pr[a \succ_v W],$$
  which is a contradiction to $W$ being a CWS in the weighted-prefox election.
\end{proof}

This can be generalized to show that any bound on $\alpha$-undominated sets transfers to the case of $\chiv_f(T) < 1/\alpha$ for tournaments.

\begin{lemma}\label{lem:tourn-undom}
If every election has an $\alpha$-undominated set of size $k$, then every tournament with $\chiv_f(T) < \frac{1}{\alpha}$ satisfies $\outdom(T) \leq k$.
\end{lemma}

\begin{proof}
  Using the solution for \eqref{LP-frac}, we construct the weighted-prefix election on $T$.  Let $\chiv_f(T) = c$.  This election has an $\alpha$-undominated set $W$ of size $k$. If $W$ is not a dominating set of $T$, then $W$ has a common in-neighbor $u$ in $T$. By Lemma \ref{clm:common},
       \begin{eqnarray*}
         \alpha < \frac{1}{c} \leq \Pr[u \succ_v W].
             \end{eqnarray*}
    This is a contradiction to $W$ being $\alpha$-undominated, so $W$ is a dominating set of $T$ of size $k$.
\end{proof}

In fact, ChatGPT tells us (without an elegant proof) that if $\chiv_f(T) < 2$, then $\outdom(T) \leq 2$.  (This implies that if $\chiv_f(T) \leq 2$, then $\outdom(T) \leq 3$.)  So the contrapositive of Lemma \ref{lem:tourn-cws-dom} does not give us a path to disprove Conjecture \ref{conj:cws3}.

\section{Domination on Digraphs} \label{sec:magic}

To obtain an equivalence, we turn to general digraphs.
For a digraph $D$, a vertex $y$ is a {\em strict in-neighbor} of $x$
if arc $yx$ belongs to $D$ and arc $xy$ does not belong to $D$.  A
{\emph {magic}} dominating set of $D$ is a set $X \subset V(D)$ such
that every vertex not in $X$ has a strict in-neighbor or non-neighbor
in $X$. 
Formally, the {\emph{magic domination}} number of $D$ is
\begin{align*}
\magic(D):= \min \sum_{u \in V(D)} x_u & \nonumber\\
\sum_{v \notin N^+(u)} x_v \geq 1 & \quad \forall u \in V(D) \nonumber\\
x_u \in \{0,1\}. & \nonumber \tag{M}
\label{IP-magic}
\end{align*}
Notice that for a tournament, the magic domination number is just the domination number.  Moreover, for an oriented graph (a digraph with no digons, where a {\emph{digon}} is a directed 2-cycle), a vertex can be dominated by an in-neighbor or an ``invisible'' link.  In a general digraph, a vertex can be dominated by a strict in-neighbor or by an ``invisible'' link; a vertex cannot be dominated by another vertex connected by a digon.

\begin{definition}
If every set of at most $k$ vertices in $D$ has a common in-neighbor, then we say that {\emph{the property $P_k(D)$ is satisfied}}.
  \end{definition}

Having a magic domination number of at most $k$ is equivalent to having a vertex set of size at most $k$ with no common in-neighbor.

\begin{lemma}\label{lem:common-in}
$P_k(D) \iff \magic(D) \geq k+1.$
\end{lemma}

\begin{proof}
First we show $P_k(D) \Rightarrow \magic(D) \geq k+1$.  Suppose $X = \{x_1, \ldots, x_k\}$ belongs to a magic dominating set.  Let $h(X)$ be a common in-neighbor of $X$.  Then there must be at least one more vertex in the magic dominating set, specifically a strict in-neighbor or a non-neighbor of $h(X)$.  Hence, $\magic(D) \geq k+1$.

Next we show that if $P_k(D)$ does not hold, then $\magic(D) \leq k$.  If $P_k(D)$ does not hold, then there is some set $X = \{x_1, \ldots, x_k\}$ such that $X$ has no common in-neighbor.  Then every vertex in $V(D)\setminus{X}$ has a strict in-neighbor or a non-neighbor in $X$, so $\magic(D) \leq k$.
\end{proof}

\subsection{Equivalence of Magic Domination and Condorcet Winning Sets}

We now show the equivalence of the magic domination number in an
oriented graph (i.e., a digraph with no digons) and the size of a CWS.  Notice that a digraph
with fractional dichromatic number less than 2 does not contain any
digons.

\begin{theorem}\label{main:equivCWS}
Every weighted election has a CWS of size at most $k$ iff every
digraph $D$ with $\chiv_f(D) < 2$ satisfies $\magic(D) \leq k$.
  \end{theorem}

\begin{proof}
First, we show the $\Rightarrow$ direction.  Let $D$ be a digraph with
$c = \chiv_f(D) < 2$.  We construct a weighted-prefix election
$\Ed(D)$ for $D$.  We assume this election has a CWS $W$ of size at
most $k$.  We will show that $W$ is a magic dominating set for $D$.
If not, then $W$ has a common in-neighbor $u$ in $D$.
By Lemma \ref{clm:common}, $\Pr[u \succ_v W] \geq 1/c > 1/2$.
This contradicts the assumption
that $W$ is a CWS for $\Ed(D)$.

Now we show the other direction: If any digraph $D$ with $\chiv_f(D) <
2$ satisfies $\magic(D) \leq k$, then every election has a CWS of size
$k$.  We begin with a weighted election $\Ed$ and construct a digraph
$D(\Ed)$ as follows.  For every candidate $a$ and every set of voters
$A$ covering more than half the weight, we make a vertex $(a,A)$.  We
add an arc from $(a,A)$ to $(b,B)$ if $a$ beats $b$ for all voters $v
\in A$.

Notice that if $\{a_1, \ldots, a_k\}$ is not a CWS in $\Ed$, then the
vertices $(a_1,A_1), \ldots, (a_k, A_k)$ have a common in-neighbor
$(u,U)$, where vertex $u$ beats $\{a_1, \ldots, a_k\}$ on the voter
set $U$.

  \begin{claim}
$\chiv_f(D(\Ed)) < 2$.
  \end{claim}

  \begin{cproof}
Consider the set $F_v$ which includes all vertices $(a, A)$ where $v
\in A$.  Notice that there are $m$ such sets.  We claim that $D[F_v]$
is acyclic for every voter $v$. Indeed, we observe that all arcs
present in $D[F_v]$ are forward with respect to the ranking of voter
$v$ and thus cannot have any cycles.  Give weight $w_v$ to each
acyclic set $D[F_v]$.  Then every vertex $(a,A)$ is covered by more
than $\frac 1 2$ of the weight, since $A$ corresponds to a set of
voters whose weight is more than half the weight of all the voters.
    \end{cproof}

Now suppose that $\magic(D(\Ed)) \leq k$.  This means that there are $k$ vertices in $D(\Ed)$ that have no common in-neighbor in $D(\Ed)$, call them $(a_1,A_1), \ldots, (a_k, A_k)$.  Hence, as previously argued, $\{a_1, \ldots, a_k\}$ is a CWS for $\Ed$.
  \end{proof}

This result can be generalized in a natural way to $\alpha$-undominated sets and digraphs with $\chiv_f(D)<\frac 1\alpha$. Note that when $\chiv_f(D) \geq 2$, $D$ can contain digons.  So the generalization concerns directed graphs, not just oriented graphs.
\begin{restatable}{theorem}{thmGenEquiv} \label{th:general-equivalence}
    Every weighted election has an $\alpha$-undominated set of size $k$ iff every digraph $D$ with $\chiv_f(D) < \frac1\alpha$ satisfies $\magic(D) \leq k$.
  \end{restatable}
The proof of Theorem \ref{th:general-equivalence} is very similar to the proof of Theorem \ref{main:equivCWS} and can therefore be found in Appendix \ref{sec:app-proofs}.

\section{Upper Bounds on Magic Domination}\label{sec:magic-main}

The notion of magic domination gives us an alternative way to approach the problem of bounding Condorcet winning sets and $\alpha$-undominated sets.  In this section, we prove bounds on the magic domination number
without assuming the existence of $\alpha$-undominated sets for some $\alpha$.  In fact, due to the equivalence shown in Theorem \ref{th:general-equivalence}, we give an independent proof, based on fractional dichromatic number, of Theorem \ref{thm:snl26}.

\subsection{Polynomial Bound via Fixed Points}\label{sec:polyBound}

In the initial version of this paper, we presented the first polynomial
bound for domination on tournaments in terms of its fractional
dichromatic number.  The proof also applies to the case of magic
domination in digraphs.  (Recall that magic domination and domination
are equivalent for tournaments.)  The main tool is Brouwer's
fixed-point theorem.  The proof was found by AI based on a prompt
inspired by \cite{prasanna2026blog}.

Notice that the approach using weighted-prefix elections (e.g., Lemma
\ref{lem:poly-tournois} and Theorem \ref{th:linear}) can also be used
to prove Theorem \ref{thm:quasilinear}.  However, the proof presented
next avoids this connection.  Moreover, it is this approach that was
extended by \cite{kim2026tournamentdichromaticnumber10} to obtain
another proof of a linear bound in the case of tournaments.  This
extension also applies to magic domination.

\begin{theorem}\label{thm:quasilinear}
If digraph $D$ satisfies $\chiv_f(D) \leq c$, then  $\magic(D) \leq \lceil 2c \ln{(2c)} \rceil$. 
\end{theorem}

\begin{proof}
Let $\{\zz_S\}$ be an optimal feasible solution to the fractional
  dichromatic number (i.e., an optimal solution
  to \eqref{LP-frac}). Hence, for each set $S$ for which $\zz_S > 0$,
  $D[S]$ is an induced acyclic subdigraph of $D$.  By assumption, we
  have \begin{eqnarray}
\sum_{S \in \SST} \zz_S \leq c,
        \end{eqnarray}
  where $\SST$ is the set of all vertex sets that induce acyclic subdigraphs in $D$.

  \begin{claim}\label{clm:maxSet}
For any probability distribution $p$ on $V(D)$, there is an $S \in \SST$ such that $p(S) \geq \frac{1}{c}$.\footnote{For $X \subset V(D)$, we define $p(X) := \sum_{v \in X}p(x)$.} 
\end{claim}

\begin{cproof}
Let $p_{\max}$ denote $\max_{S \in \SST} p(S)$. Then
  \begin{eqnarray*}
p_{max} \cdot c \geq p_{\max} \sum_{S \in \SST} \zz_S \geq \sum_{S \in \SST} \zz_S \cdot p(S) =
\sum_{S \in \SST} \zz_S \sum_{v \in S} p(v) =
 \sum_v p(v) \sum_{S: v \in S} \zz_S \geq  \sum_v p(v) = 1.
    \end{eqnarray*}
The last inequality follows from the main constraint of
\eqref{LP-frac}.  Thus, there exists some $S \subset V$, where $D[S]$
is an induced acyclic subdigraph, such that $p(S) \geq \frac{1}{c}$.
\end{cproof}

Set $k = \magic(D) -1$ and assume $k \geq 1$.  Then every set of at
most $k$ vertices has a common in-neighbor (by Lemma
\ref{lem:common-in} since it is not a magic dominating set).  For
every ordered $k$-tuple $(x_1, \ldots, x_k)$, fix a vertex $h(x_1,
\ldots, x_k)$ that is a common in-neighbor of this $k$-tuple (i.e.,
beats these $k$ vertices).

For a probability distribution $p$ on $V(D)$, let $F(p)$ be the
probability distribution of $h(X_1, \ldots, X_k)$, where $X_i$ are
independent with distribution $p$.  This is a continuous map from the
probability simplex to itself, so it has a fixed point, which we will
refer to as $\pp$ from now on.

Denote the vertex set whose existence is implied by Claim \ref{clm:maxSet} as $S'$ and let
$\ell = |S'|$.  Set $a = p(S') \geq 1/c$.  Order $S'$ transitively,
$v_1, \ldots, v_{\ell}$, all arcs going from left to right.  Let $j$ be
the first index for which the prefix has mass at least $a/2$.  Set
$B = \{v_1, \ldots, v_j\}$ and $C = \{v_j, \ldots, v_{\ell}\}$.  Then
$p(B) \geq a/2$ and $p(C) > a/2$.

Now we will draw $\{X_1, \ldots, X_k\}$ independently from distribution $p$.  This will also yield a $Y$ distributed according to $p$ with $Y \rightarrow X_i$ for all $i$.  If $Y \in C$, then none of the $X_i$'s can lie in $B$, since $Y$ beats all $X_i$'s.  Therefore,

\begin{eqnarray}
\frac{a}{2} < p(C) \leq \Pr[X_1,\ldots,X_k \notin B] = (1-p(B))^k \leq (1-a/2)^k.
  \end{eqnarray}
Using $a \geq 1/c$, we have
\begin{eqnarray}
\frac{1}{2c} < \left(1 - \frac{1}{2c} \right)^k \leq e^{-k/{2c}}.
  \end{eqnarray}
Hence, $k < 2c \ln{(2c)}$.  Since $\magic(D) = k+1$ is an integer, this proves the claimed bound.
\end{proof}

\subsection{Linear Bound via \cite{kim2026tournamentdichromaticnumber10}}

\cite{kim2026tournamentdichromaticnumber10} presented a clever modification of the proof of Theorem \ref{thm:quasilinear} to provide another proof of a linear bound in the case of tournaments.  We observe that their modification also applies in the case of magic domination, and we show their constant can be tightened to match the currently best-known constant for undominated sets.

\begin{theorem}\label{th:magic-bound}
If digraph $D$ satisfies $\chiv_f(D) \leq c$, then $\magic(D) \leq \KKc \chiv_f(D)-1$.
\end{theorem}

Following the notation of \cite{kim2026tournamentdichromaticnumber10}, 
for a set of vertices $X$, let $M(X)$ be the common in-neighborhood of $X$,
$$
    M(X) = \bigcap_{x\in X}N^-(x).
    $$
    The proof of the next lemma is essentially the same as Lemma 4.3 in \cite{kim2026tournamentdichromaticnumber10}.
    
\begin{lemma}\label{lem:KO}
  Let $D$ be a digraph with $\chiv_f(D) \leq c$.  For every $\ell$ and any $0< \delta < 1$, there is a set $X$ of at most $\ell$ vertices of $D$ such that
    $$
        \chiv_f(M(X)) \leq \frac c \delta \left(1-\frac{1-\delta}c\right)^\ell.
    $$
\end{lemma}

\begin{proof}
  Define $G(\ell,\delta) := \frac c \delta \left(1-\frac{1-\delta}c\right)^\ell$.
  For a contradiction, suppose that for any set $X$ of size at most $\ell$, $\chiv_f(D[M(X)]) > G(\ell,\delta)$. By duality of the fractional chromatic number, there is a probability distribution on vertices of $M(X)$, call it $\mu_X$, such that every induced acyclic set of $D[M(X)]$ has weight at most $\frac 1{G(\ell,\delta)}$. We can extend the domain of $\mu_X$ to the rest of $V(D)$ by setting $\mu_X(v)=0$ to all vertices not in $M(X)$.
  
    For a probability distribution $p$, we sample $\ell$ times from $p$ to get a set $X$, and we pick a vertex according to $\mu_X$. This gives a new probability distribution that we call $F_p$.
    $$
        F_p(v) = \Ex_{X\sim p^\ell}[\mu_X(v)].
    $$
        $F$ is a continuous map from the probability simplex to itself, so according to Brouwer, it has a fixed point $p^*$. There is an acyclic set $S$ such that $p^*(S)\geq \frac1c$. Give a topological order to $S$ and call it $\{v_1,\dots, v_s\}$. There is a threshold index $j$ such that $L=\{v_1,\dots, v_j\}$ has weight $p^*(L)\geq \frac {1-\delta}{c}$ and $R=\{v_j,\dots v_s\}$ has weight $p^*(R) > \frac {\delta} c$.

    Let $Y$ be the vertex sampled according to $\mu_X$, with $X$ sampled from $(p^*)^\ell$.  Suppose that $Y$ is in $R$. Then, none of the vertices in $X$ belong to $L$ since $\{v_1,\dots, v_s\}$ is a topological order and $Y$ is in the common in-neighborhood of $X$. Since $p^*$ is a fixed point, $Y$ also follows $p^*$. So the probability of sampling $Y$ in $R$ is bounded by the probability of sampling $X$ outside $L$, followed by sampling someone from $R$ in $\mu_X$.
    $$
        \frac{\delta} c < p^*(R) = \Ex_{X\sim (p^*)^\ell}[\mu_X(R)\cdot \mathbbm 1[X\cap L = \emptyset]]\leq \frac 1 {G(\ell,\delta)}(1-p^*(L))^\ell \leq \frac 1 {G(\ell,\delta)}\left(1-\frac{1-\delta}c\right)^\ell.
    $$
    And therefore,
    $$
        G(\ell,\delta) < \frac c \delta \left(1-\frac{1-\delta}c\right)^\ell.
    $$
        This is a contradiction, and we conclude that the desired set $X$ exists.
\end{proof}

Now we are ready to prove Theorem \ref{th:magic-bound}.

\begin{proof}[Proof of Theorem \ref{th:magic-bound}]
  Let $\rho=3.51\dots$ be the solution to $e^{\rho-1} = \rho^2$, and $\eta = \frac {\rho^2}{\rho-1} = \KK$.  We proceed by induction on the number of vertices of $D$. The theorem is true when $1\leq |V(D)|\leq2$, as $\magic(D)\leq 2 < \eta\cdot 1-1$.  So assume $|D|\geq 3$ and let $c=\chiv_f(D)$.  We can apply Lemma \ref{lem:KO} with
    $$
        \ell = \lfloor \rho c \rfloor \geq \rho c - 1, \quad \delta = \frac 1 \rho.
    $$
    Then, with our choice of $\rho$,
    $$
        G\left(\lfloor \rho c \rfloor, \frac 1\rho\right) \leq c\rho \left(1 - \frac {\rho - 1}{\rho c}\right)^{\rho c-1} \leq c \rho  e^{-(\rho - 1)} = \frac c \rho.
    $$
    The lemma gives us a set $X$ of size at most $\rho c$ such that $\chiv_f(D[M(X)])\leq\frac c\rho$. Note that $X$ cannot be empty, otherwise $\chiv_f(D[M(X)]) = \chiv_f(D) > \frac c \rho$. So, $|M(X)|<|D|$ and we can apply the induction hypothesis to $M(X)$. We obtain $Z$, a magic dominating set of $M(X)$, of size at most $\eta\cdot \chiv_f(D[M(X)])-1$. The set $X\cup Z$ is a magic dominating set of $D$, and
    $$
        |X\cup Z|\leq |X| + |Z| \leq \rho c + \eta \frac c \rho -1 = (1-\frac 1\rho)\eta c + \eta \frac c \rho -1 = \eta c - 1.
    $$
\end{proof}

\section{Lower Bounds for Magic Domination}\label{sec:graph-theory}

In this section, we explore what is known for bounds on magic domination in digraphs as a function of the fractional dichromatic number.  In the previous sections, we have shown connections and equivalences between these graph theory problems and Condorcet winning sets and $\alpha$-undominated sets in elections.
Due to the equivalences shown in Section \ref{sec:magic}, we obtain lower bounds for magic domination in digraphs.

\begin{theorem} [\cite{charikar2025six}, due to Zbarsky]\label{bound-charikar}
If $\alpha < \frac 2 {k+1}$, then there exist elections with no $\alpha$-undominated committees of size $k$.
\end{theorem}
This together with Theorem \ref{th:general-equivalence} implies the following theorem.
\begin{theorem} \label{th:magic-lb}
    For every integer $k \geq 1$, if $c > \frac {k+1} 2$, then there exists a digraph $D$ with $\chiv_f(D) = c$ satisfying $\magic(D) \geq k+1$.
\end{theorem}

\cite{charikar2025six} also stated the following conjecture, which states that Theorem \ref{bound-charikar} is tight.

\begin{conjecture}[\cite{charikar2025six}] \label{conj:undominated-set}
For any $k$, if $\alpha \geq \frac{2}{k+1}$, then every election has an $\alpha$-undominated set of size $k$.
\end{conjecture}

This is a generalization of Conjecture \ref{conj:cws3}, which states
that every election has a Condorcet winning set of size three.  An
analogous conjecture is that Theorem \ref{th:magic-lb} is tight for
magic domination.

\begin{conjecture} \label{conj:magic-dom}
Every digraph $D$, $\magic(D) \leq  \lfloor 2\chiv_f(D) \rfloor$.
\end{conjecture}

\begin{lemma}
Conjectures \ref{conj:undominated-set} and \ref{conj:magic-dom} are equivalent.
  \end{lemma}

\begin{proof}
Indeed, Theorem \ref{th:general-equivalence} gives an equivalence between Conjecture \ref{conj:undominated-set} and the statement\\
``For every digraph $D$ and every $k$, if $\chiv_f(D)< \frac {k+1}2$, then $\magic(D)\leq k$.''

We show that this formulation (i) is equivalent to Conjecture \ref{conj:magic-dom} (ii). First, (i)$\implies$(ii). $\lfloor 2\chiv_f(D)\rfloor$ is integer and $\chiv_f(D)<\frac {\lfloor 2\chiv_f(D)\rfloor +1}2$. So, by (i), $\magic(D)\leq \lfloor 2\chiv_f(D)\rfloor$.\\
(ii)$\implies$(i). Let $k\in \mathbb N$. Then, if $\chiv_f(D)<\frac {k+1}2$, then $2\chiv_f(D)<k+1$, and $\lfloor 2\chiv_f(D)\rfloor < k+1$. But, since $k+1$ is integer, we get that $\lfloor 2\chiv_f(D)\rfloor \leq k$ and so by (ii), $\magic(D)\leq k$.
\end{proof}

In order to disprove Conjecture \ref{conj:magic-dom}, it suffices to find a digraph with ratio $\magic(D)/\chiv_f(D) > 2$ or a tournament such that $\outdom(T)/\chiv_f(T) > 2$.  To study the feasibility of finding such a digraph or tournament, we summarize the known bounds for digraphs in Table \ref{tab:magic-bounds} and Figure \ref{fig:magic-plot}.

\begin{table}[ht]
    \centering
    \begin{tabular}{|c|c|c|}
        \hline
        $k$ & Max value of $\chiv_f(D)$ that ensures $\magic(D) \leq k$ & Min value of $\chiv_f(D)$ where $\magic(D)$ can be $> k$ \\
         & Theorems \ref{th:bound-small-k} and \ref{th:general-equivalence} & Theorem \ref{th:magic-lb}\\
        \hline
        2 & 1.33333... & 1.5\\
        \hline
        3 & 1.62575... & 2\\
        \hline
        4 & 1.89562... & 2.5\\
        \hline
        5 & 2.15050... & 3\\
        \hline
        6 & 2.39438... & 3.5\\
        \hline
        7 & 2.62972... & 4\\
        \hline
    \end{tabular}
    \caption{The current best bounds on the fractional dichromatic number for a fixed magic domination value. In the left column, if a digraph $D$ has $\chiv_f(D)$ smaller than the entry, then $\magic(D) \leq k$.  (For example, if $\chiv_f(D) < 4/3$, then $\magic(D) \leq 2$.)
      These bounds are obtained by transferring the result of Theorem \ref{th:bound-small-k} to magic domination using Theorem \ref{th:general-equivalence}. The right column indicates that there are digraphs with $\chiv_f(D)$ greater than the entry, for which $\magic(D) > k$. If the value of $\chiv_f(D)$ lies between the two entries, then we currently do not know whether or not the magic domination number is bounded by $k$.
      }
    \label{tab:magic-bounds}
\end{table}

\begin{figure}[ht]
    \centering
    \includegraphics[scale=0.55]{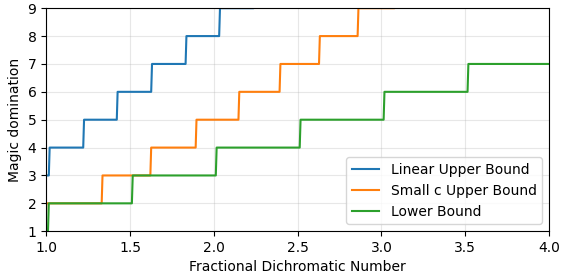}
    \caption{Plot of the current bounds for magic domination with respect to the fractional dichromatic number. The blue line ``Linear Upper Bound" is obtained from Theorem \ref{th:magic-bound}, the orange line ``Small $c$ Upper Bound" is obtained from Theorems \ref{th:bound-small-k} and \ref{th:general-equivalence}, and the green line ``Lower Bound" is obtained from Theorem \ref{th:magic-lb}.}
    \label{fig:magic-plot}
\end{figure}

\subsection{Tournaments Differ from Digraphs}

The existence of a tournament with $\outdom(T)/\chiv_f(T)>2$ is unlikely.
ChatGPT did give a proof that $\chiv_f(T) < 2$ implies $\outdom(T) \leq 2$, but the proof is inscrutable; we currently do not have an understandable proof of this statement.  However, we can show a weaker statement.

\begin{theorem} \label{th:tournament2}
  Every tournament with $\chiv_f(T)< \frac 7 4$ satisfies $\outdom(T) \leq 2$.
\end{theorem}

\begin{proof}
    Start from the weighting on the acyclic sets from the fractional dichromatic number solution and give the weight of an acyclic set $S$ to its source vertex $\sigma(S)$.
    $$
        w_v = \sum_{S\in \mathcal S, v=\sigma(S)}z_S, \quad \quad \sum_v w_v = \chiv_f(T) = c.
    $$
    Notice that every set covering $v$ has its source in $N^-[v]$, so $w(N^-[v]) \geq 1$, and $w(N^+(v)) \leq c-1$. We can use Claim \ref{clm:maxSet} to find an acyclic set $A$ of weight at least $1$. Let $a$ be the source of $A$. Since $A\setminus \{a\}$ is included in $N^+(a)$,
    $$
        w_a \geq w(A) - w(N^+(a))\geq 2-c.
    $$
    So, if you remove every set with source $a$ from the fractional dichromatic number solution, every vertex in $N^-(a)$ is still covered by at least $1$. Therefore,
    $$
        \chiv_f(T[N^-(a)])\leq c-w_a \leq 2c-2.
    $$
    With $c<\frac 74$, the right side is strictly less than $\frac 32$, and so is acyclic. Therefore, picking $a$ and the source of $N^-(a)$ gives a dominating set of size $2$.\end{proof}

Theorem \ref{th:tournament2} is stronger than the statement implied by Theorem \ref{th:general-equivalence} assuming the existence of $\frac{2}{3}$-undominated sets of size 2.  Consequently, a tournament cannot give a stronger lower bound for undominated sets of size 2.

On the other end, there are tournaments with domination number $3$ and fractional dichromatic number approaching $2$ from above. Figure \ref{fig:tournament} represents a family of such tournaments $A_m$. Any pair of vertices has a common in-neighbor, so the domination number is at least (but in fact equal to) $3$. The fractional dichromatic number is $\frac {2m+1}{m}$.

\begin{figure}[ht]
    \centering
    \begin{tikzpicture}[
    vertex/.style={
        circle, fill=black,
        inner sep=1.8pt
    },
    edge/.style={
        -{Stealth[length=2mm,width=1.5mm]},
        shorten <=0.2cm, shorten >=0.2cm,
        thin
    },
    thickedge/.style={
        -{Stealth[length=2mm,width=1.5mm]},
        shorten <=0.2cm, shorten >=0.2cm
    }
]
    \def\M{3}          
    \def\R{2}          

    \pgfmathtruncatemacro{\N}{2*\M+1}
    \pgfmathtruncatemacro{\last}{\N-1}
    \pgfmathtruncatemacro{\lastoffset}{\M-1}

    \foreach \i in {0,...,\last} {
        \pgfmathsetmacro{\ang}{90-360*\i/\N}
        \pgfmathtruncatemacro{\lab}{\i+1}

        \coordinate (v-\i) at
            ({\R*cos(\ang)},{\R*sin(\ang)});
    }

    \foreach \i in {1,...,\last} {

        \foreach \d in {1,...,\lastoffset} {
            \pgfmathtruncatemacro{\j}{mod(\i+\d,\N)}
            \ifthenelse{\j=0}{}{
            \draw[edge] (v-\i) -- (v-\j);}
        }

        \pgfmathtruncatemacro{\j}{mod(\i+\M+1,\N)}
        \ifthenelse{\j=0}{}{
        \draw[edge] (v-\i) -- (v-\j);}
    }

    \foreach \d in {1,...,\lastoffset} {
        \draw[edge, color=red] (v-0) -> (v-\d);
    }

    \pgfmathtruncatemacro{\j}{mod(\M+1,\N)}
    \draw[edge, color=red] (v-0) -> (v-\j);
    
    \pgfmathtruncatemacro{\j}{mod(\M,\N)}
    \draw[edge, color=blue] (v-\j) -> (v-0);
    
    \foreach \d in {1,...,\lastoffset} {
        \pgfmathtruncatemacro{\j}{mod(\N - \d,\N)}
        \draw[edge, color=blue] (v-\j) -> (v-0);
    }

    \foreach \i in {0,...,\last} {
        \pgfmathsetmacro{\ang}{90-360*\i/\N}
        \pgfmathtruncatemacro{\lab}{\i+1}

        \node[vertex,
            label={[font=\small]\ang:$v_{\lab}$}]
            at (v-\i) {};
    }

\end{tikzpicture}
    \caption{The tournament $A_3$. The tournament $A_m$ has $2m+1$ vertices.
      Vertex 1 has an arc to vertices 2 through $m$ and to vertex $m+2$.  All other vertices have an arc to vertex 1.
 The arcs colored in red are out-going arcs of $1$, and the blue arcs are in-going arcs. This pattern is repeated for every vertex. The tournament is vertex transitive, and any set of $m$ consecutive vertices is acyclic.}
    \label{fig:tournament}
\end{figure}

We believe that the ratio between (magic) domination and fractional dichromatic number might be smaller for tournaments than
for digraphs. We formulate the following conjecture.

\begin{conjecture}
Every tournament $T$ satisfies $\outdom(T) \leq \lfloor{\chiv_f(T)}\rfloor +1.$
  \end{conjecture}

In the case where $\chiv_f(T)=2$, it is less clear what the bound on the domination should be.  It seems to be the case that $\chiv_f(T) \leq 2$ implies
$\outdom(T) \leq 3$ (by the same inscrutable proof referenced above), but we have no specific example of a tournament with $\chiv_f(T)=2$ and $\outdom(T) = 3$.
Therefore, the relationship could very well be following the stronger conjecture.
\begin{conjecture}
$\outdom(T) \leq \lceil{\chiv_f(T)}\rceil.$
  \end{conjecture}

\subsection{Bounds on Digraphs}

We also looked for digraphs whose ratio $\magic(D)/\chiv_f(D)$ exceeds $2$, in particular, digraphs with $\chiv_f(D) < \frac 32$ and $\magic(D)=3$.

\begin{figure}[ht]
    \centering
    \begin{minipage}[c]{0.48\textwidth}
        \centering
        \begin{tikzpicture}[
    ->,
    >=stealth,
    vertex/.style={
        circle, fill=black,
        inner sep=2pt
    },
    every edge/.style={draw, thick}
]
    \def\n{10} 
    \def\r{2} 

    \foreach \i in {1,...,\n} {
        \pgfmathsetmacro{\phi}{90-360*(\i-1)/\n}
        \pgfmathsetmacro{\x}{\r*cos(\phi)}
        \pgfmathsetmacro{\y}{\r*sin(\phi)}
        \pgfmathsetmacro{\lx}{(\r+0.4)*cos(\phi)}
        \pgfmathsetmacro{\ly}{(\r+0.4)*sin(\phi)}

        \node[vertex]
                (v\i) at (\x,\y) {};
        \node (c\i) at (\lx, \ly) {$v_{\i}$};
    }

    \foreach \i in {1,...,\numexpr\n-1\relax} {
        \pgfmathtruncatemacro{\j}{\i+1}
        \draw[->] (v\i) to[bend left=10] (v\j);
    }

    \draw[->] (v\n) to[bend left=10] (v1);
\end{tikzpicture}
        \caption{The digraph $C_{10}$. Any set $V-v_i$ is acyclic, and any two vertices are a magic dominating set.}
        \label{fig:C_10}
    \end{minipage}
    \hfill
    \begin{minipage}[c]{0.48\textwidth}
        \centering
        \begin{tikzpicture}[
    vertex/.style={
        circle, fill=black,
        inner sep=1.5pt
    },
    inner edge/.style={
        ->, thick, >=stealth,
        shorten <=2pt, shorten >=2pt
    },
    outer edge/.style={
        ->, thick, >=stealth,
        double, double distance=3pt,
        shorten <=0.9cm, shorten >=0.9cm,
        -{Stealth[length=8pt,width=10pt]}
    }
]
    \def\N{3}       
    \def\M{5}       
    \def\R{2.0}     
    \def\r{0.52}    

    \foreach \i in {1,...,\N} {
        \pgfmathsetmacro{\ang}{90-360*(\i-1)/\N}
        \pgfmathsetmacro{\cx}{\R*cos(\ang)}
        \pgfmathsetmacro{\cy}{\R*sin(\ang)}

        \coordinate (c-\i) at (\cx,\cy);
        \draw[dashed, white!50] (c-\i) circle[radius=0.75];

        \pgfmathsetmacro{\lx}{(\R+1.0)*cos(\ang)}
        \pgfmathsetmacro{\ly}{(\R+1.0)*sin(\ang)}
        \node at (\lx,\ly) {$C^{(\i)}$};
    }

    \foreach \i in {1,...,\N} {
        \pgfmathtruncatemacro{\ip}{mod(\i,\N)+1}
        \draw[outer edge] (c-\i) -- (c-\ip);
    }

    \foreach \i in {1,...,\N} {
        \pgfmathsetmacro{\ang}{90-360*(\i-1)/\N}
        \pgfmathsetmacro{\cx}{\R*cos(\ang)}
        \pgfmathsetmacro{\cy}{\R*sin(\ang)}

        \foreach \j in {1,...,\M} {
            \pgfmathsetmacro{\phi}{90-360*(\j-1)/\M}
            \pgfmathsetmacro{\x}{\cx+\r*cos(\phi)}
            \pgfmathsetmacro{\y}{\cy+\r*sin(\phi)}

            \node[vertex] (v-\i-\j) at (\x,\y) {};
        }
    }

    \foreach \i in {1,...,\N} {
        \foreach \j in {1,...,\M} {
            \pgfmathtruncatemacro{\jp}{mod(\j,\M)+1}
            \draw[inner edge] (v-\i-\j) to[bend left=15] (v-\i-\jp);
        }
    }
\end{tikzpicture}
        \caption{The digraph $B_5$, which is the blow-up of a $C_3$ with a $C_5$. The maximum induced acyclic sets consist of any pair of cycles with a vertex removed from each one.}
        \label{fig:B_5}
    \end{minipage}
\end{figure}
Figures \ref{fig:C_10} and \ref{fig:B_5} represent two classes of digraphs that approach a ratio of $2$ from below. A cycle $C_m$ has magic domination $2$ and fractional dichromatic number $\frac {m}{m-1}$. The digraph $B_m$ is made up of three cycles $C^{(1)}$, $C^{(2)}$ and $C^{(3)}$ with $m$ vertices each, and each vertex of $C^{(i)}$ dominates every vertex of $C^{(i+1)}$. These digraphs have magic domination $3$, and fractional dichromatic number $\frac {3m}{2(m-1)}$. They are interesting, because when constructing their respective prefix-weighted elections, we obtain the cyclic election for $C_m$ and an election very similar to the Kronecker product of cyclic elections of a 3-cycle and an $m$-cycle, which are the best known lower bound examples for the undominated set problem \cite{elkind2015condorcet}.

These two digraphs can be used to construct a new digraph with $\magic(D)= k+1$ and $\chiv_f(D)$ approaching $\frac {k+1}2$ for any $k$. Indeed, we can define the digraph product $A\graphProd B$ made from two digraphs $A$ and $B$, then adding a digon between any vertex of $A$ to any vertex of $B$. This operation behaves nicely with the fractional dichromatic number and magic domination.

\begin{lemma}
    For any two graphs $A$ and $B$,
    \begin{itemize}
        \item $\chiv_f(A\graphProd B) = \chiv_f(A) + \chiv_f(B)$,
        \item $\magic(A\graphProd B) = \magic(A) + \magic(B)$.
    \end{itemize}
\end{lemma}

\begin{proof}
    The proof of each point follows directly from the construction.
    \begin{itemize}
        \item No acyclic set can contain vertices from both $A$ and $B$, so covering $A\graphProd B$ requires $A$ and $B$ to be covered separately. The equality holds because taking an optimal cover for $A$ and for $B$ is valid for $A\graphProd B$.
        \item No vertex of $A$ magically dominates a vertex of $B$, or the other way around. So $A$ and $B$ need to be covered separately. The equality also holds.
    \end{itemize}
\end{proof}


This gives a more direct proof of Theorem \ref{th:magic-lb}.

\begin{proof}[Proof of Theorem \ref{th:magic-lb}]
  If $k$ is odd, say $k=2\ell+1$, then $L_m = C^{(1)}_m\graphProd \dots \graphProd C^{(\ell+1)}_m$ has magic domination $\magic(L_m) = 2\ell +2 = k+1$, and fractional dichromatic number $\chiv_f(L_m)= (\ell+1)\frac m{m-1}=\frac{(k+1)m}{2(m-1)}$ which approaches $\frac{k+1}2$ as $m$ grows.
  
    If $k$ is even, say $k=2\ell$, then $L_m = C^{(1)}_m\graphProd \dots \graphProd C^{(\ell-1)}_m\graphProd B_m$ has $\magic(L_m) = 2(\ell -1) +3 = k+1$ and $\chiv_f(L_m) = (\ell-1)\frac m{m-1} + \frac {3m}{2(m-1)}=\frac{(k+1)m}{2(m-1)}$.
\end{proof}

Notice that a digraph $D$ with $\chiv_f(D) < \frac{3}{2}$ contains no cyclic triangles.  And if $\magic(D) \geq 3$, then every pair of vertices have a common in-neighbor.
The problem of the existence of such a digraph was introduced by \cite{myersphd}
and answered positively by \cite{charbit2005circuits,AHLNVW2015}.
However, these constructions appear to have fractional dichromatic number greater than $3/2$.  (The smallest working example of the ``toric digraph'' defined in Section 5.2, pages 62--63, of \cite{charbit2005circuits} 
has $\chiv_f(D)\geq \frac {16} 9$ according to ChatGPT.)
Therefore, a special case of this problem is the following.
\begin{question}
Let $D$ be a digraph where every pair of vertices has a common in-neighbor. Can $\chiv_f(D) < \frac 32$?
\end{question}

If such a digraph exists, it would disprove the lower bound stated in Conjecture \ref{conj:magic-dom} and give a new lower bound for $\alpha$-undominated sets of size $2$ in elections.

\section{Local-to-Global Consequences}\label{sec:localGlobal}

Observation \ref{obs:fractBounded}, Theorem \ref{th:magic-bound} and Observation \ref{obs:tBounded}
imply the following quadratic bound on the dichromatic number of
$t$-bounded tournaments.

\begin{theorem}\label{thm:tBounded}
Every $t$-bounded tournament $T$ satisfies $\chiv(T) \leq 9.83 t^2$.
\end{theorem}

\subsection{$t$-Arc-Bounded Tournaments}

Let $\Delta(uv) := N^+(v) \cap N^-(u)$ be the vertices in $V(T)$ that
make a directed triangle with arc $uv$.  We say a tournament is {\em
  $t$-arc-bounded} if for every arc $uv$, the subtournament induced on
the vertex set $\Delta(uv)$ has dichromatic number at most $t$.  (We always assume such tournaments are strongly connected.)
The
property of being $1$-arc-bounded was used in
\cite{klingelhoefer2024coloring} to design algorithms for coloring
2-dicolorable tournaments.  Like $t$-bounded tournaments,
$t$-arc-bounded tournaments were also shown to have dichromatic number
bounded by a function depending only on
$t$~\cite{klingelhoefer2024bounding,NSS24}.
In fact, for $t$-arc-bounded tournaments, we can asymptotically match the bound
stated for $t$-bounded tournaments in Theorem \ref{thm:tBounded}.
\begin{theorem}\label{thm:main-t-arc-bounded}
Every $t$-arc-bounded tournament satisfies $\chiv(T) \leq 30t^2 +3t$.
  \end{theorem}
Before we prove this, we make another observation relating $t$-bounded
and $t$-arc-bounded tournaments.  Recall that $\outdom(T)$ denotes the
size of the smallest dominating set of $T$, and similarly $\indom(T)$
denotes the size of the minimum absorbing set in $T$, where $S \subset
V(T)$ is an absorbing set of $T$ if every vertex in $T$ either belongs
to $S$ or has at least one out-neighbor in $S$.  By reversing arc
directions, the upper bound on $\outdom(T)$ stated in Theorem
\ref{thm:main-new} also holds for $\indom(T)$.

\begin{lemma}\label{lem:arcToVertex2}
In a $t$-arc-bounded tournament, there are vertices $s_1,s_2$ such that 
\begin{itemize}
  
\item[(i)]  $\chiv(T[N^+(s_1)]) \leq t \cdot (\indom(T) - 1)$, and

  \item[(ii)] $\chiv(T[N^-(s_2)]) \leq t \cdot (\outdom(T) - 1)$.
\end{itemize}
\end{lemma}

\begin{proof}
Let $X$ be a dominating set of size $\outdom(T)$.  Then $T[X]$ has a
Hamilton path.  Let $s_2$ be the first vertex in this path.  Consider
a vertex $u \in N^-(s_2)$ that is not in $X$.  Then $u$ has an
in-neighbor in $X$ and hence $u$ belongs to $\Delta(ab)$ for some arc
$ab$ on the path.  Thus, we can color these vertices using $\outdom(T)
- 1$ arc neighborhoods, including vertices that belong to $X$.  (To
see this, notice that a vertex $u \in N^-(s_2)$ has an in-neighbor
which is its predecessor on the Hamilton path.)  So we can use at most
$t \cdot(\outdom(T)-1)$ colors to color $N^-(s_2)$.  To prove (i), we
use the same argument, letting $s_1$ be the last vertex of a Hamilton
path on an absorbing set.
  \end{proof}

The proof of the next lemma was found by \chat and can be found in Appendix \ref{sec:app-proofs}.
\begin{restatable}{lemma}{FracThree}\label{lem:frac3t}
In a $t$-arc-bounded tournament, there is a vertex $s$ such that $\chiv_f(T[N^+(s)] \leq 3t$.
  \end{restatable}

\begin{lemma}\label{lem:t-arc-dom}
Every $t$-arc-bounded tournament satisfies $\outdom(T) \leq 15t$.
\end{lemma}

\begin{proof}
Using Lemma \ref{lem:frac3t}, we find a vertex $s$ such that $\chiv_f(T[N^-(s)] \leq 3t$.  We have that $\outdom(T) \leq 15t$, since $s$ itself dominates its out-neighborhood and its in-neighborhood requires $15t-1$ vertices by Theorem \ref{th:magic-bound}.   
\end{proof}

\begin{proof}[Proof of Theorem \ref{thm:main-t-arc-bounded}]
We can apply Lemmas \ref{lem:arcToVertex2} and \ref{lem:t-arc-dom} to find vertices $s_1$ and $s_2$ such that $\chiv(T[N^+(s_1)]), \chiv(T[N^-(s_2)] \leq 15t^2-t$.  Applying Lemma 2.6 of \cite{klingelhoefer2024coloring} shows that we can color $T$ with $30t^2 + 3t$ colors.
\end{proof}

\vspace{3mm}

The following lemma is also useful, since it implies the following corollary.
\begin{corollary}\label{cor:tBound}
A $t$-arc-bounded tournament is a $(5 t\cdot \indom(T))$-bounded tournament. 
  \end{corollary}

\begin{lemma}\label{lem:arcToVertex}
Let $T$ be a $t$-arc-bounded tournament.  For every $v \in V(T)$, we have
\begin{itemize}
  
\item[(i)]  $\chiv(T[N^-(v)]) \leq 5 t \cdot \outdom(T)$, and

  \item[(ii)] $\chiv(T[N^+(v)]) \leq 5 t \cdot \indom(T)$.

\end{itemize}
  
  \end{lemma}

\begin{proof}
Recall that $\Delta(uv)$ is the vertex set in $V(T)$ that
make a directed triangle with arc $uv$. Let $\Phi(uv) := N^+(u)\cap
N^-(v)$ be the vertices in $V(T)$ that occur in the middle of two-arc
paths from $u$ to $v$ for an arc $uv$.  Since $T$ is $t$-arc-bounded,
for any arc $uv$ in $T$, we have $\chiv(T[\Delta(uv)]) \leq t$.  Using
the tools from \cite{klingelhoefer2024coloring}, we show that we can
also color $\Phi(uv)$ with few colors.

\begin{claim}\label{clm:forwardNeighborhood}
An arc $uv$ in a $t$-arc-bounded tournament $T$ satisfies $\chiv(T[\Phi(uv)]) \leq 5t$.
  \end{claim}

\begin{cproof}
Let $P$ be a shortest path from $v$ to $u$ in $T$ and let $V(P)$ be
its vertices (including $u$ and $v$) and $A(P)$ its arcs.  Consider
the subtournament of $T$ induced on vertex set $V'=\Phi(uv) \cup
V(P)$.  Then every vertex in $\Phi(uv)$ belongs to $\Delta(e)$ for
some arc $e$ in $A(P)$.  By Lemma 2.6 in
\cite{klingelhoefer2024coloring}, the tournament $T[V']$ can be
colored with at most $5t$ colors.
  \end{cproof}

Let $X$ be a minimum dominating set of $T$, and let $u$ be a vertex not in $X$.  Then there is an arc $xu$ for some vertex $x \in X$.  If all vertices in $N^-(u)$ are in $X$, then (i) holds, since then there are at most $|X|$ vertices in $N^-(u)$.  So suppose there is a vertex $v \in N^-(u)$ for $v \notin X$.  Then there is some $y \in X$ such that arc $yv$ is in $T$.  Then there are two cases: if arc $uy \in A(T)$, then $v \in \Delta(uy)$, while if $yu \in A(T)$, then $v \in \Phi(yu)$.  Applying Claim \ref{clm:forwardNeighborhood}, we conclude that $T[N^-(u)]$ can be colored with at most $5t \cdot |X| =  5t \cdot \outdom(T)$ colors.  An analogous argument can be used to prove (ii).~\end{proof}

\section{Statement of AI Usage}\label{sec:aiusage}

We used both the \chat model and a free chat-GPT account.  The latter was only used to generate image templates (that were edited).

The \chat model produced the proof of Theorem \ref{thm:quasilinear} in
Section \ref{sec:magic-main} after uploading \cite{prasanna2026blog}
in the prompt and asking it for a proof of Theorem \ref{thm:main}
using fixed points.  This proof gave the first polynomial bound for
Theorem \ref{thm:main}.  The book proof in \cite{prasanna2026blog}
crucially uses Brouwer's fixed-point theorem in a particularly
elegant and natural way that seems tailored to domination problems.
(As another example, it also produced an improved bound for the main
theorem of \cite{bousquet2019proof}, using essentially the same idea
as in the Proof of Theorem \ref{thm:quasilinear} and the book proof in
\cite{prasanna2026blog}.  See Appendix \ref{sec:ssw} for this proof.)

The main ideas in Section \ref{sec:tournaments} and \ref{sec:magic}
were obtained with the help of \chat, but were not autonomously
produced by it.  In Section \ref{sec:graph-theory}, it suggested the
direct proof of Theorem \ref{th:magic-lb} and produced a proof of
Theorem \ref{th:tournament2}, which used some of our ideas.  In
Section \ref{sec:localGlobal}, Lemma \ref{lem:frac3t} was obtained by
\chat.  We did not use AI to write the paper, but we did use it to
find and fix small errors.

\bibliographystyle{alpha}
\bibliography{tournaments}

\appendix

\section{Missing Proofs}\label{sec:app-proofs}

\thmGenEquiv*

\begin{proof}
We show the $\Rightarrow$ direction first.  Let $D$ be a digraph with
$\chiv_f(D) < \frac1\alpha$. We construct the weighted-prefix election
$\Ed(D)$ for $D$. By assumption, there is an $\alpha$-undominated set
$W$ of size $k$ in this election.  If $W$ is not a magic dominating
set in $D$, then $W$ has a common in-neighbor $u$ in $D$.  By Lemma
\ref{clm:common}, $\Pr[u \succ_v W] \geq 1/c > \alpha$.  This
contradicts the fact that $W$ is $\alpha$-undominated for $\Ed(D)$.
  
Now we prove the $\Leftarrow$ direction.  We begin with a weighted
election $\Ed$ and construct a digraph $D(\Ed)$ as follows. For every
candidate $a$ and every set of voters $A$ covering more than an
$\alpha$-fraction of the weight, we make a vertex $(a,A)$.  We add an
arc from $(a,A)$ to $(b,B)$ if $a$ beats $b$ for all voters $v \in
A$. This digraph may have digons (e.g., there can be a digon between
$(a,A)$ and $(b,B)$ if $A$ and $B$ are disjoint sets of voters such
that all voters in $A$ rank $a$ before $b$ and all voters in $B$ rank
$b$ before $a$).  We show that this digraph has fractional chromatic
number less than $\frac 1 \alpha$.

Notice that if $\{a_1, \ldots, a_k\}$ is not an $\alpha$-undominated set in $\Ed$, then the vertices $(a_1,A_1), \ldots, (a_k, A_k)$ in $D(\Ed)$ have a common in-neighbor $(u,U)$, where vertex $u$ beats $\{a_1, \ldots, a_k\}$ on the voter set $U$.

    \begin{claim}
      $\chiv_f(D(\Ed)) < \frac1 \alpha.$
      \end{claim}
      
\begin{cproof}
    Consider the set $F_v$ which includes all vertices $(a, A)$ where voter $v \in A$. We claim that $D[F_v]$ is acyclic for every voter $v$. Indeed, we observe that all arcs present in $D[F_v]$ are forward with respect to the ranking of voter $v$ and thus cannot have any cycles (or any digons).  Give weight $w_v$ to each acyclic set $D[F_v]$.  Then every vertex $(a,A)$ is covered by more than $\alpha$-fraction of the weight, since $A$ corresponds to a set of voters whose weight is more than an $\alpha$-fraction of the weight of all the voters. So $D(\Ed)$ has fractional chromatic number less than $\frac 1 \alpha$.
    \end{cproof}
    
Therefore, by assumption, $D(\Ed)$ has a magic dominating set $S$ of
size at most $k$.  This means that there are $k$ vertices in $D(\Ed)$ that have no common in-neighbor in $D(\Ed)$, call them $(a_1,A_1), \ldots, (a_k, A_k)$.  Hence, as previously argued, $\{a_1, \ldots, a_k\}$ is an $\alpha$-undominated set for $\Ed$.
\end{proof}

\vspace{5mm}

The next lemma and proof was found by \chat during an extended conversation on related topics.
\FracThree*

\begin{proof}
  We define $q$ to be a probability distribution on $V(T)$ such that for every vertex $v \in V(T)$, we have
  \begin{eqnarray}
    q(N^+(v)) \geq q(N^-(v)),
  \end{eqnarray}
  which implies
  \begin{eqnarray}
    q(N^+[v]) \geq \frac{1 + q(v)}{2} \geq \frac{1}{2}. \label{line2}
  \end{eqnarray}
  
Choose a vertex $u$ such that $q(u) > 0$.  Set
$$U = N^+(u), \quad B = N^-(u), \quad q(U) = q(B) = \frac{1-q(u)}{2} < \frac{1}{2}.$$

\begin{claim}\label{clm:outneighbor}
Every $s \in U$ has an out-neighbor $z \in B$.  
\end{claim}

\begin{cproof}
  Otherwise, all out-neighbors of $s$ would lie in $U$, and hence $N^+[s] \subseteq U$.  This would contradict \eqref{line2}.
  \end{cproof}

For each vertex $s \neq u$, we consider a shortest path from $s$ to $u$.  By Claim \ref{clm:outneighbor}, this path has one or two arcs.  For each arc $ab$ on the path, we color $\Delta(ab) \cap U$ with $t$ colors and give each color class weight $2q(s)$.

\begin{claim}
The total weight of this fractional coloring is $3t$.
\end{claim}

\begin{cproof}
  For each vertex $s \in B$, the path contains one arc and hence the total contribution is $2t \cdot q(B)$.  For each vertex $s \in U$, the path contains two arcs and hence the total contribution is $4t \cdot q(U)$.  So the total value is at most
\begin{eqnarray*}
  2t \cdot q(B) + 4t \cdot q(U) ~ \leq ~ 2t \cdot q(U) + 4t \cdot q(U) ~ = ~ 6t \cdot q(U) ~ < ~ 3t. 
    \end{eqnarray*}\end{cproof}

Now we need to argue that this fractional coloring covers each vertex $x \in U$.
For each $x \in U$, consider $y \in N^+[x]$.  If $y$ has an arc to $u$, then $x$ received $2q(y)$.  If $y$ has no arc to $u$, then we have triangle $ybu$ for some $b \in B$.  In this case, $x$ belongs to either $\Delta(yb)$ or $\Delta(bu)$, and thus also received $2q(y)$.  So $x$ is covered by $2q(N^+[x]) \geq 1$.
\end{proof}

\section{Fixed-Point Proof of Sands-Sauer-Woodrow}\label{sec:ssw}

We say a tournament $T$ is {\em $k$-arc-colored} if each arc has at
least one of $k$ colors.  Say $u \leadsto_i v$ when there is a path
from $u$ to $v$ using color $i$ (or, trivially, if $u = v$).  When
there is a color $i \in [k]$ such that $u \leadsto_i v$, we say that
there is a {\em monochromatic path from $u$ to $v$}.

\cite{sands1982monochromatic} asked if there is a function $f$ such that for every $k$-arc-colored tournament $T$, there is a vertex set $S$ of size at most $f(k)$ with the following property:  every vertex in $T$ can be reached via a monochromatic path from a vertex in $S$.   \cite{sands1982monochromatic} gives a proof that $f(2) = 1$.

\cite{bousquet2019proof} answered this question positively by showing that $f(k) \leq O(\ln{(2k)}\cdot k^{k+2})$.  A nice proof, found by \chat, after uploading \cite{prasanna2026blog} and requesting a proof using fixed points, gives a better bound.  Note that a lower bound by \cite{beaudou2019lower} is $2^k/\sqrt{k}$.  

\begin{theorem}\label{thm:SSW}
  Let $T$ be a $k$-arc-colored tournament. There is a set of vertices $S \subset V(T)$, with $|S|\leq O(2^kk^2)$, 
  from which every vertex can be reached via a monochromatic path.
\end{theorem}

\begin{proof}
  For each color $i \in [k]$, we define
$$ N_i^-[v] := \{u:u\leadsto_i v\}, \quad N_i^+[v] := \{w:v\leadsto_i w\}.$$
Also, define $\delta := \frac{1}{k \cdot 2^{k+1}}$.

\begin{claim}\label{clm:good-vertex}
For any probability distribution $p$ on the $V(T)$, there is a vertex $v$ and a color $i$ such that
    $$p(N_i^-[v])\geq\delta \quad \text{ and } \quad p(N_i^+[v])\geq\delta.$$
\end{claim}

\begin{cproof}
Assume this is false.  Then for every vertex $v$ and color $i$, at least one of $N_i^+[v]$ or $N_i^-[v]$ that has weight strictly less than $\delta$.
For each vertex, record which neighborhood is small for each color (choose either if both are small). We can partition the vertices into sets that share the same records. There are $2^k$ possible sets, so one of them, say $W$, has weight $$s:=p(W)\geq 2^{-k}.$$
For each color, define
    $$M_i := \sum_{u,v \in W: {u\leadsto_i v}}p(u)p(v).$$
    Suppose that the in-neighborhoods of color $i$ are small.  In other words, $$
    p(N^-_i[u] \cap W) < \delta \text{ for all } u \in W.$$  Then
    $$
        M_i = \sum_{v\in W}p(N_i^-[v]\cap W)p(v) < \sum_{v\in W}\delta p(v) = \delta s.
    $$
        If the out-neighborhoods of color $i$ are small, then sum over the first vertex $u$ instead of $v$, then we have
$$
        M_i = \sum_{v\in W}p(N_i^+[u]\cap W)p(u) < \sum_{v\in W}\delta p(v) = \delta s.
    $$
        Therefore,
$$s^2 = \sum_{u,v \in W} p(u)p(v) \leq 2 \sum_{i=1}^k M_i < 2k \delta s.$$
(Notice that the 2 is there because each arc is contained in the previous sum twice.)    Dividing by $s$, we have
    $$
        s < 2k\delta = 2^{-k}.
    $$
    This is a contradiction, and thus proves the claim.~\end{cproof}

Suppose no set of size $m$ is a dominating set by monochromatic paths.
Then for every $m$-tuple $S = (v_1, v_2, \ldots, v_m)$, there is a vertex $x$ for which there is no monochromatic path from any vertex in $S$ to $x$.  For every $m$-tuple, fix such a vertex and call it $h(S)$.  Observe that a probability distribution on $V(T)$ yields a probability distribution on $V(T)$ via $h$ and hence has a fixed point $p$.

Now consider the vertex $v$ and color $i$ guaranteed by Claim \ref{clm:good-vertex}. 
Let $B = N_i^-[v]$ and $C = N_i^+[v]$.  Let us choose
$$X_1, X_2, \ldots, X_m \sim p,$$ which implies $Y \sim p$ where $Y = h(X_1,\ldots,X_m)$.

Notice that the probability of choosing $Y$ from $C$ is $p(C)$.  By the transitivity of reachability on arcs of color $i$, $Y \in C$ requires that every $X_j$ is not in $B$.  So we have
\begin{eqnarray*}
\delta \leq p(C) \leq (1-p(B))^m \leq (1-\delta)^m \leq e^{-\delta m}.
\end{eqnarray*}
Hence we can upper bound $m$ as claimed. $$m \leq \frac{1}{\delta} \ln{(\frac{1}{\delta})} = 2^{k+1} k \cdot \ln{(2^{k+1} k)} = O(2^k k^2).$$~\end{proof}

\end{document}